\documentclass[leqno,12pt]{amsart} 
\usepackage{amssymb}
\theoremstyle{plain} 
\newtheorem{theorem}{\indent\sc Theorem}[section]
\newtheorem{lemma}[theorem]{\indent\sc Lemma}
\newtheorem{corollary}[theorem]{\indent\sc Corollary}
\newtheorem{proposition}[theorem]{\indent\sc Proposition}

\theoremstyle{definition} 

\newtheorem{remark}[theorem]{\indent\sc Remark}

\def\C{{\mathbf{C}}}
\def\R{{\mathbf{R}}}
\def\H{{\mathbf{H}}}
\def\N{{\mathbf{N}}}
\def\Pi{{\mathbf{P}}}
\def\Si{{\mathbf{S}}}
\def\Lc{{\mathcal{L}}}

\def\tr#1{\mathord{\mathopen{{\vphantom{#1}}^t}#1}}

\begin{document}

\title[On Exceptional values of Gauss maps]{On the maximal number of exceptional values of Gauss maps for 
various classes of surfaces} 

\author[Y. Kawakami]{Yu Kawakami} 



\renewcommand{\thefootnote}{\fnsymbol{footnote}}
\footnote[0]{2010\textit{ Mathematics Subject Classification}.
Primary 30D35, 53C42 ; Secondary 30F45, 53A10, 53A15.}
\thanks{ 
Partly supported by the Grant-in-Aid for Young Scientists (B) No.~24740044, Japan Society for the Promotion of Science.}
\address{
Graduate School of Science and Engineering, \endgraf
Yamaguchi university, \endgraf
Yamaguchi, 753-8512, Japan
}
\email{y-kwkami@yamaguchi-u.ac.jp}



\maketitle

\begin{abstract}
The main goal of this paper is to reveal the geometric meaning of the maximal number of exceptional values 
of Gauss maps for several classes of immersed surfaces in space forms, for example, complete minimal surfaces in the Euclidean three-space, 
weakly complete improper affine spheres in the affine three-space and weakly complete flat surfaces in the hyperbolic three-space. 
For this purpose, we give an effective curvature bound for a specified conformal metric on an open Riemann surface. 
\end{abstract}

\section*{Introduction} 
The geometric nature of value distribution theory of complex analytic mappings is well-known. 
One of the most elegant results of the theory is the geometric meaning of the precise maximum ``$2$'' for the number of exceptional values 
of nonconstant meromorphic functions on the complex plane $\C$. Here we call a value that a function or map never assumes an exceptional value 
of the function or map. In fact, Ahlfors \cite{Ah1935} and Chern \cite{Ch1960} showed 
that the least upper bound for the number of exceptional values of nonconstant holomorphic maps from $\C$ to a closed Riemann surface coincides with the 
Euler number of the closed Riemann surface by using Nevanlinna theory (see also \cite{Ko2003}, \cite{no1990} and \cite{Ru2001}). 
In particular, for nonconstant meromorphic functions on $\C$, the geometric meaning of the maximal number ``$2$'' of exceptional values 
is the Euler number of the Riemann sphere. We note that if the closed Riemann surface is of genus $\geq 2$, then such a map does not exist 
because the Euler number is negative. 

On the other hand, global properties of the Gauss map of complete minimal surfaces in the Euclidean three-space ${\R}^{3}$ 
are closely related to value-distribution-theoretic properties of meromorphic functions on $\C$. 
In particular, Fujimoto \cite{Fu1988} proved that the precise maximum for the number of exceptional values of the Gauss map of 
a nonflat complete minimal surface in ${\R}^{3}$ is ``$4$'', and Osserman \cite{Os1964} showed that the Gauss map of a nonflat algebraic minimal 
surface can omit at most $3$ values (by an algebraic minimal surface, we mean a complete minimal surface with finite total curvature). 
Recently, the author, Kobayashi and Miyaoka \cite{KKM2008} gave an effective upper bound for the number of 
exceptional values of the Gauss map for a special class of complete minimal surfaces that includes algebraic minimal surfaces (this class is called 
the pseudo-algebraic minimal surfaces). This also provided a geometric meaning for the Fujimoto and Osserman results for this class, 
because the upper bound is described in terms of geometric invariants. 
However, from \cite{KKM2008} it was still not possible to understand the geometric meaning for general class. 

The author also investigated value-distribution-theoretic properties of Gauss maps for several classes of surfaces which may admit singularities. 
For instance, by refining the Fujimoto analytic argument, the author and Nakajo \cite{KN2012} showed that the maximal number of exceptional values of 
the Lagrangian Gauss map of weakly complete improper affine fronts in the affine three-space ${\R}^{3}$ is ``$3$''. 
As an application of this result, a simple proof of the parametric affine Bernstein theorem for improper affine spheres in ${\R}^{3}$ was provided. 
Moreover, the authors \cite{Ka2012, KN2012} proved similar results for flat fronts in the hyperbolic three-space ${\H}^{3}$. 

The aim of this paper is to reveal the geometric meaning of the precise maximum for the number of exceptional values of Gauss maps for 
these classes of surfaces. The paper is organized as follows: In Section 1, we give a curvature bound for the conformal metric 
$ds^{2}=(1+|g|^{2})^{m}|\omega|^{2}$ on an open Riemann surface $\Sigma$, where $\omega$ is a holomorphic $1$-form and $g$ is a 
meromorphic function on $\Sigma$ (Theorem \ref{main-1}). The proof is given in Section 2. As a corollary of this theorem, 
we prove that the precise maximum for the number of exceptional values of the nonconstant meromorphic function $g$ on $\Sigma$ with the complete conformal 
metric $ds^{2}$ is ``$m+2$'' (Corollary \ref{omit-1} and Proposition \ref{Voss}). We note that the geometric meaning of the ``$2$'' in ``$m+2$'' is 
the Euler number of the Riemann sphere (Remark \ref{Ahl-Chern}). In Section $3$, we give some applications of the main results. In particular, 
we give the geometric meaning of the maximal number of exceptional values of Gauss maps for several classes of immersed surfaces in space forms. 
For instance, the induced metric from ${\R}^{3}$ on complete minimal surfaces is $ds^{2}=(1+|g|^{2})^{2}|\omega|^{2}$ (i.e. $m=2$), 
thereby the maximal number of exceptional values of the Gauss map $g$ of nonflat complete minimal surfaces in ${\R}^{3}$ is ``$4$ $(=2+2)$''. 
On the other hand, for the Lagrangian Gauss map $\nu$ of weakly complete improper affine fronts, since $\nu$ is meromorphic, $dG$ is holomorphic 
and the complete metric is $d{\tau}^{2}=2(1+|\nu|^{2})|dG|^{2}$ (i.e. $m=1$), 
the maximal number of exceptional values of the Lagrangian Gauss map of weakly complete improper affine fronts in ${\R}^{3}$ is ``$3$ $(=1+2)$''. 

Finally, the author would like to thank Professors Junjiro Noguchi, Wayne Rossman, Masaaki Umehara, Kotaro Yamada and the referee 
for their useful advice and comments. In addition, the author would like to express his thanks to Professors Ryoichi Kobayashi, 
Masatoshi Kokubu, Miyuki Koiso and Reiko Miyaoka for their encouragement of this study. 


\section{Main results}
Now we state the main theorem of this paper. 

\begin{theorem}\label{main-1}
Let $\Sigma$ be an open Riemann surface with the conformal metric 
\begin{equation}\label{conformal}
ds^{2}=(1+|g|^{2})^{m}|\omega|^{2},
\end{equation}
where $\omega$ is a holomorphic $1$-form, $g$ is a meromorphic function on $\Sigma$, and $m\in {\N}$. 
Suppose that $g$ omits $q\geq m+3$ distinct values . 
Then there exists a positive constant $C$, depending on $m$ and the set of exceptional values, 
but not the surface $\Sigma$, such that for all $p \in \Sigma$  
we have 
\begin{equation}\label{curvature-estimate}
|K_{ds^{2}}(p)|^{1/2}\leq \dfrac{C}{d(p)}, 
\end{equation}
where $K_{ds^{2}}(p)$ is the Gaussian curvature of the metric $ds^{2}$ at $p$ and $d(p)$ is the geodesic distance from $p$ to  
the boundary of $\Sigma$, that is, the infimum of the lengths of the divergent curves in $\Sigma$ emanating from $p$. 
\end{theorem}

As a corollary of Theorem \ref{main-1}, we give the following Picard-type theorem for the meromorphic function $g$ on $\Sigma$ 
with the complete conformal metric $ds^{2}=(1+|g|^{2})^{m}|\omega|^{2}$. 

\begin{corollary}\label{omit-1}
Let $\Sigma$ be an open Riemann surface with the conformal metric given by (\ref{conformal}). 
If the metric $ds^{2}$ is complete and the meromorphic function $g$ is nonconstant, then $g$ can omit at most $m+2$ distinct values. 
\end{corollary}

\begin{proof}
By way of contradiction, assume that $g$ omits $m+3$ distinct values. If $ds^{2}$ is complete, then we may set $d(p)= \infty$ for all $p\in \Sigma$. 
By virtue of Theorem \ref{main-1}, $K_{ds^{2}}\equiv 0$ on $\Sigma$. On the other hand, the Gaussian curvature of the metric $ds^{2}$ is 
given by 
\begin{equation}
K_{ds^{2}}=-\dfrac{2m|g_{z}'|^{2}}{(1+|g|^{2})^{m+2}|{\hat{\omega}}_{z}|^{2}}, 
\end{equation}
where $\omega ={\hat{\omega}}_{z}dz$ and $g_{z}'=dg/dz$. 
Thus $K_{ds^{2}}\equiv 0$ if and only if $g$ is constant. This contradicts the assumption that $g$ is nonconstant. 
\end{proof}

\begin{remark}\label{Ahl-Chern}
The geometric meaning of the ``$2$'' in ``$m+2$'' is the Euler number of the Riemann sphere. 
Indeed, if $m=0$ then the metric $ds^{2}=(1+|g|^{2})^{0}|\omega|^{2}=|\omega|^{2}$ is flat and complete on $\Sigma$. 
We thus may assume that $g$ is a meromorphic function on ${\C}$ because $g$ is replaced by $g\circ \pi$, where $\pi\colon \C\to \Sigma$ is 
a holomorphic universal covering map. On the other hand, Ahlfors \cite{Ah1935} and Chern \cite{Ch1960} showed that the best possible upper bound ``$2$'' of the number of 
exceptional values of nonconstant meromorphic functions on $\C$ coincides with the Euler number of the Riemann sphere. 
Hence we get the conclusion. 
\end{remark}

Corollary \ref{omit-1} is optimal because there exist the following examples. 
\begin{proposition}\label{Voss}
Let $\Sigma$ be either the complex plane punctured at $q-1$ distinct points ${\alpha}_{1}, \cdots, {\alpha}_{q-1}$ or the universal cover of 
that punctured plane. 
We set 
$$
\omega=\dfrac{dz}{{\prod}_{i=1}^{q-1}(z-{\alpha}_{i})}, \quad g=z. 
$$
Then $g$ omits $q$ distinct values and the metric $ds^{2}=(1+|g|^{2})^{m}|\omega|^{2}$ is complete if and only if $q\leq m+2$. 
In particular, there exist examples whose metric $ds^{2}$ is complete and $g$ omits $m+2$ distinct values. 
\end{proposition}

\begin{proof}
We can easily show that $g$ omits the $q$ distinct values ${\alpha}_{1}, \cdots, {\alpha}_{q-1}$ and $\infty$ on $\Sigma$. 
A divergent curve $\Gamma$ in $\Sigma$ must tend to one of the points ${\alpha}_{1}, \cdots, {\alpha}_{q-1}$ or $\infty$.   
Thus we have 
$$
\int_{\Gamma}ds=\int_{\Gamma}(1+|g|^{2})^{m/2}|\omega|=\int_{\Gamma}\dfrac{(1+|z|^{2})^{m/2}}{\prod_{i=1}^{q-1}|z-{\alpha}_{i}|}|dz|= \infty, 
$$
when $q\leq m+2$. 
\end{proof}

\section{Proof of the main theorem}
We first recall the notion of chordal distance between two distinct values in the Riemann sphere $\C\cup \{\infty \}$. 
For two distinct values $\alpha$, $\beta\in \C\cup \{\infty\}$, we set 
$$
|\alpha, \beta|:= \dfrac{|\alpha -\beta|}{\sqrt{1+|\alpha|^{2}}\sqrt{1+|\beta|^{2}}}
$$
if $\alpha \not= \infty$ and $\beta \not= \infty$, and $|\alpha, \infty|=|\infty, \alpha| := 1/\sqrt{1+|\alpha|^{2}}$. 
We note that, if we take $v_{1}$, $v_{2}\in {\Si}^{2}$ with $\alpha =\varpi (v_{1})$ and $\beta = \varpi (v_{2})$, we have that 
$|\alpha, \beta|$ is a half of the chordal distance between $v_{1}$ and $v_{2}$, where $\varpi$ denotes the stereographic projection of 
the $2$-sphere ${\Si}^{2}$ onto $\C\cup \{\infty \}$. 

Before proceeding to the proof of Theorem \ref{main-1}, we recall two function-theoretical lemmas. 
\begin{lemma}{\cite[(8.12) on page 136]{Fu1997}}\label{Lemma2-1}
Let $g$ be a nonconstant meromorphic function on ${\Delta}_{R}=\{z\in \C ; |z|< R \}$ $(0<R\leq +\infty)$ which omits $q$ values 
${\alpha}_{1}, \ldots, {\alpha}_{q}$. If $q>2$, then for each positive $\eta$ with $\eta <(q-2)/q$, then there exists a positive constant $C'$, 
depending on $q$ and $L:=\min_{i< j}|{\alpha}_{i}, {\alpha}_{j}|$, such that 
\begin{equation}\label{Lemma2-1-1}
\dfrac{|g'_{z}|}{(1+|g|^{2})\prod_{j=1}^{q}|g, {\alpha}_{j}|^{1-\eta}}\leq C'\dfrac{R}{R^{2}-|z|^{2}}. 
\end{equation}  
\end{lemma}

\begin{lemma}{\cite[Lemma 1.6.7]{Fu1993}}\label{Lemma2-2} 
Let $d{\sigma}^{2}$ be a conformal flat metric on an open Riemann surface $\Sigma$. 
Then, for each point $p\in \Sigma$, there exists a local diffeomorphism $\Phi$ of a 
disk ${\Delta}_{R}=\{z\in \C ; |z|< R \}$ $(0<R\leq +\infty)$ onto an open 
neighborhood of $p$ with $\Phi (0)=p$ such that $\Phi$ is a local isometry, that is, 
the pull-back ${\Phi}^{\ast}(d{\sigma}^{2})$ is equal to the standard Euclidean metric $ds^{2}_{Euc}$ on ${\Delta}_{R}$ 
and, for a point $a_{0}$ with $|a_{0}|=1$, the ${\Phi}$-image ${\Gamma}_{a_{0}}$ of the curve $L_{a_{0}}=\{w:= a_{0}s ; 0 < s < R\}$ 
is divergent in $\Sigma$. 
\end{lemma}

\begin{proof}[{\it Proof of Theorem \ref{main-1}}] 
Let ${\alpha}_{1}, \ldots, {\alpha}_{q}$ be the exceptional values of $g$. We may assume that ${\alpha}_{q}= \infty$ 
after a suitable M\"obius transformation. 
We choose a positive number $\eta$ with 
$$
\dfrac{q-2(m+1)}{q}< \eta < \dfrac{q-(m+2)}{q}
$$
and set $m\not= 0$ and $\lambda := m / (q-2-q\eta)$. Since $q \geq m+3$, then $1/2< \lambda < 1$ holds. 

Now we define a new metric 
\begin{equation}\label{proof1-1}
d{\sigma}^{2} = |{\hat{\omega}}_{z}|^{2/(1-\lambda)}\biggl{(}\dfrac{1}{|g'_{z}|} 
{\prod}_{j=1}^{q-1}\biggl{(}\dfrac{|g-{\alpha}_{j}|}{\sqrt{1+|{\alpha}_{j}|^{2}}} \biggr{)}^{1-\eta} \biggr{)}^{2\lambda /(1- \lambda)} |dz|^{2} 
\end{equation}
on the set ${\Sigma}' :=\{p\in \Sigma \,;\, g'_{z}(p)\not= 0 \}$, where $\omega =\hat{\omega}_{z} dz$ and $g'_{z} =dg / dz$ with respect to 
the local complex coordinate $z$. Take a point $p\in {\Sigma}'$. Since the metric $d{\sigma}^{2}$ is flat, by Lemma \ref{Lemma2-2}, there 
exists a local isometry $\Phi$ satisfying $\Phi (0)=p$ from a disk ${\Delta}_{R}=\{z\in \C ; |z|< R\}$ $(0< R \leq +\infty)$ with the 
standard Euclidean metric $ds^{2}_{Euc}$ onto an open neighborhood of $p \in {\Sigma}'$ with the metric $d{\sigma}^{2}$, such that, 
for a point $a_{0}$ with $|a_{0}|=1$, the $\Phi$-image ${\Gamma}_{a_{0}}$ of the curve $L_{a_{0}}=\{w:=a_{0}s ; 0 < s < R \}$ is divergent in ${\Sigma}'$. 
For brevity, we denote the function $g\circ \Phi$ on ${\Delta}_{R}$ by $g$ in the following. By Lemma \ref{Lemma2-1}, we get 
\begin{equation}\label{proof1-2}
R\leq C'\dfrac{1+|g(0)|^{2}}{|g'_{z}(0)|}\prod_{j=1}^{q}|g(0), {\alpha}_{j}|^{1-\eta} < +\infty.
\end{equation}
Hence 
$$
L_{d\sigma}({\Gamma}_{a_{0}})=\int_{{\Gamma}_{a_{0}}}d\sigma =R< +\infty ,
$$
where $L_{d\sigma}({\Gamma}_{a_{0}})$ denotes the length of ${\Gamma}_{a_{0}}$ with respect to the metric $d{\sigma}^{2}$. 
We assume that the ${\Phi}$-image ${\Gamma}_{a_{0}}$ tends to a point $p_{0}\in \Sigma \backslash {\Sigma}'$ as $s\to R$. 
Taking a local complex coordinate $\zeta:= g'_{z}$ in a neighborhood of $p_{0}$ with $\zeta (p_{0})=0$, we can write
$$
d{\sigma}^{2}=|\zeta|^{-2\lambda /(1-\lambda)}w |d\zeta|^{2} 
$$
for some positive smooth function $w$. Since $\lambda /(1-\lambda)>1$, we have 
$$
R=\int_{{\Gamma}_{a_{0}}}d\sigma \geq \widetilde{C}\int_{{\Gamma}_{a_{0}}}\dfrac{|d\zeta|}{|\zeta|^{\lambda /(1-\lambda)}} = +\infty, 
$$
which contradicts (\ref{proof1-2}). Thus ${\Gamma}_{a_{0}}$ diverges outside any compact subset of $\Sigma$ as $s\to R$. 

Since $d{\sigma}^{2}=|dz|^{2}$, we obtain by (\ref{proof1-1}) that 
\begin{equation}\label{omega-omega}
|{\hat{\omega}}_{z}|=\biggl{(}|g'_{z}|\prod_{j=1}^{q-1}\biggl{(}\dfrac{\sqrt{1+|{\alpha}_{j}|^{2}}}{|g-{\alpha}_{j}|}\biggr{)}^{1-\eta} \biggr{)}^{\lambda}. 
\end{equation}
By Lemma \ref{Lemma2-1}, we have 
\begin{eqnarray*}
{\Phi}^{\ast}ds &=& |\hat{\omega}_{z}|(1+|g|^{2})^{m/2}|dz|  \\
                &=& \biggl{(}|g'_{z}|(1+|g|^{2})^{m/2\lambda}\prod_{j=1}^{q-1}\biggl{(}\dfrac{\sqrt{1+|{\alpha}_{j}|^{2}}}{|g-{\alpha}_{j}|}\biggr{)}^{1-\eta} \biggr{)}^{\lambda}|dz| \\
                &=& \biggl{(}\dfrac{|g'_{z}|}{(1+|g|^{2})\prod_{j=1}^{q}|g, {\alpha}_{j}|^{1-\eta}} \biggr{)}^{\lambda} |dz| \\
                &\leq & (C')^{\lambda}\biggl{(}\dfrac{R}{R^{2}-|z|^{2}} \biggr{)}^{\lambda} |dz|.
\end{eqnarray*}
Thus we have
$$
d(p)\leq \int_{{\Gamma}_{a_{0}}} ds =\int_{L_{a_{0}}} {\Phi}^{\ast} ds \leq (C')^{\lambda} \int_{L_{a_{0}}} \biggl{(}\dfrac{R}{R^{2}-|z|^{2}} \biggr{)}^{\lambda} |dz| 
\leq (C')^{\lambda}\dfrac{R^{1-\lambda}}{1-\lambda}\, (<+\infty)
$$
because $0< \lambda < 1$. Moreover, by (\ref{proof1-2}), we get that 
\begin{equation}\label{Proof1-3}
d(p)\leq \dfrac{C'}{1-\lambda}\biggl{(}\dfrac{1+|g(0)|^{2}}{|g'_{z}(0)|}\prod_{j=1}^{q}|g(0), {\alpha}_{j}|^{1-\eta} \biggr{)}^{1-\lambda}. 
\end{equation}
On the other hand, the Gaussian curvature $K_{ds^{2}}$ of the metric $ds^{2}=(1+|g|^{2})^{m}|\omega|^{2}$ is given by 
$$
K_{ds^{2}}=-\dfrac{2m|g'_{z}|^{2}}{(1+|g|^{2})^{m+2}|\hat{\omega}_{z}|^{2}}. 
$$
Thus, by (\ref{omega-omega}), we also get that
\begin{equation}\label{G-curvA}
|K_{ds^{2}}|^{1/2} = \sqrt{2m} \biggl{(}\dfrac{|g'_{z}|}{1+|g|^{2}} \biggl{)}^{1-\lambda}\prod_{j=1}^{q}|g, {\alpha}_{j}|^{(1-\eta)\lambda} . 
\end{equation}
Since $|g, {\alpha}_{j}|\leq 1$ for each $j$, we obtain that
\begin{equation}\label{Curv-est}
|K_{ds^{2}}(p)|^{1/2}d(p)\leq \dfrac{\sqrt{2m}C'}{1-\lambda} =: C. 
\end{equation}
By the definitions of $C'$ and $\lambda$, we see that $C$ is positive and depends on $m$, $q$ and $L:=\min_{i< j}|{\alpha}_{i}, {\alpha}_{j}|$.  
\end{proof}

\section{Applications}
In this section, we give several applications of our main results. 

\subsection{Gauss map of minimal surfaces in the Euclidean 3-space} 
We briefly recall some basic facts of minimal surfaces in ${\R}^{3}$. Details can be found, for example, in \cite{Fu1993} and \cite{Os1986}. 
Let $X=(x^{1}, x^{2}, x^{3})\colon \Sigma \to {\R}^{3}$ be an oriented minimal surface in ${\R}^{3}$. By associating a local complex coordinate 
$z=u+\sqrt{-1}v$ with each positive isothermal coordinate system $(u, v)$, $\Sigma$ is considered as a Riemann surface whose conformal metric 
is the induced metric $ds^{2}$ from ${\R}^{3}$. Then 
\begin{equation}\label{mini-harm}
{\bigtriangleup}_{ds^{2}}X=0
\end{equation}
holds, that is, each coordinate function $x^{i}$ is harmonic. With respect to the local complex coordinate $z=u+\sqrt{-1}v$ of 
the surface, (\ref{mini-harm}) is given by 
\begin{equation}\label{mini-del}
\bar{\partial}\partial X = 0,
\end{equation}
where $\partial = (\partial /\partial u - \sqrt{-1}\partial / \partial v)/ 2$ and $\bar{\partial}= (\partial /\partial u + \sqrt{-1}\partial / \partial v)/ 2$. 
Hence each ${\phi}_{i}:= \partial x^{i}dz$ $(i=1, 2, 3)$ is a 
holomorphic $1$-form on $\Sigma$. If we set 
\begin{equation}\label{mini-Weier}
\omega = {\phi}_{1}-\sqrt{-1}{\phi}_{2}, \quad g=\dfrac{{\phi}_{3}}{{\phi}_{1}-\sqrt{-1}{\phi}_{2}}, 
\end{equation}
then $\omega$ is a holomorphic $1$-form and $g$ is a meromorphic function on $\Sigma$. Moreover, the function $g$ coincides with the composition 
of the Gauss map and the stereographic projection from ${\Si}^{2}$ onto $\C\cup \{\infty\}$, and the induced metric $ds^{2}$ is given by 
\begin{equation}\label{mini-1st}
ds^{2}=(1+|g|^{2})^{2}|\omega|^{2}. 
\end{equation}

Applying Theorem \ref{main-1} to the metric $ds^{2}$, we can show the Fujimoto theorem for the Gauss map of minimal surfaces in ${\R}^{3}$. 

\begin{theorem}\cite[Theorem I and Corollary 3.4]{Fu1988}\label{mini-Fujimoto} 
Let $X\colon \Sigma \to {\R}^{3}$ be an oriented minimal surface whose Gauss map $g\colon \Sigma \to \C\cup \{\infty\}$ omits more than $4$ $(=2+2)$ distinct 
values. Then there exists a positive constant $C$ depending on the set of exceptional values, but not the surface, such that for all $p\in \Sigma$ the 
inequality (\ref{curvature-estimate}) holds. In particular, the Gauss map of a nonflat complete minimal surface in ${\R}^{3}$ can omit at most $4$ $(=2+2)$ values. 
\end{theorem}

\subsection{Lorentzian Gauss map of maxfaces in the Lorentz-Minkowski 3-space} 
Maximal surfaces in the  Lorentz-Minkowski 3-space ${\R}^{3}_{1}$ are closely related to minimal surfaces in ${\R}^{3}$. In this subsection 
we treat maximal surfaces with some admissible singularities, called {\it maxfaces}, as introduced by Umehara and Yamada \cite{UY2006}. 
We remark that maxfaces, 
non-branched generalized maximal surfaces in the sense of \cite{ER1992} and non-branched generalized maximal maps in the sense of \cite{IK2008} are 
all the same class of maximal surfaces. The Lorentz-Minkowski $3$-space ${\R}^{3}_{1}$ is the affine $3$-space ${\R}^{3}$ with the inner product 
$$
\langle \, , \, \rangle = -(dx^{1})^{2}+(dx^{2})^{2}+(dx^{3})^{2},
$$
where $(x^{1}, x^{2}, x^{3})$ is the canonical coordinate system of ${\R}^{3}$. We consider a fibration 
$$
p_{L}\colon {\C}^{3} \ni ({\zeta}^{1}, {\zeta}^{2}, {\zeta}^{3}) \mapsto \text{Re} (-\sqrt{-1}{\zeta}^{1}, {\zeta}^{2}, {\zeta}^{3}) \in {\R}^{3}_{1}. 
$$
The projection of null holomorphic immersions into  ${\R}^{3}_{1}$ by $p_{L}$ gives maxfaces. Here, a holomorphic map $F=(F_{1}, F_{2}, F_{3})\colon \Sigma \to {\C}^{3}$ 
is called {\it null} if $\{(F_{1})'_{z}\}^{2}+\{(F_{2})'_{z}\}^{2}+\{(F_{3})'_{z}\}^{2}$ vanishes identically, where $'=d / dz$ denotes the derivative with respect to 
a local complex coordinate $z$ of $\Sigma$. For maxfaces, an analogue of the Enneper-Weierstrass representation formula is known (see also \cite{Ko1983}). 

\begin{theorem}{\cite[Theorem 2.6]{UY2006}}\label{max-EW}
Let $\Sigma$ be a Riemann surface and $(g, \omega)$ a pair consisting of a meromorphic function and a holomorphic $1$-form on $\Sigma$ such that 
\begin{equation}\label{max-nullmet}
d{\sigma}^{2}:= (1+|g|^{2})^{2}|\omega|^{2}
\end{equation}
gives a (positive definite) Riemannian metric on $\Sigma$, and $|g|$ is not identically $1$. Assume that 
$$
\text{Re} \int_{\gamma} (-2g, 1+g^{2}, \sqrt{-1}(1-g^{2}))\, \omega = 0
$$
for all loops $\gamma$ in $\Sigma$. Then 
\begin{equation}\label{max-immer}
f= \text{Re} \int^{z}_{z_{0}} (-2g, 1+g^{2}, \sqrt{-1}(1-g^{2}))\, \omega
\end{equation}
is well-defined on $\Sigma$ and gives a maxface in ${\R}^{3}_{1}$, where $z_{0}\in {\Sigma}$ is a base point. Moreover, all maxfaces are obtained 
in this manner. The induced metric $ds^{2}:=f^{\ast} \langle \, , \, \rangle$ is given by
$$
ds^{2} = (1-|g|^{2})^{2}|\omega|^{2}, 
$$
and the point $p\in \Sigma$ is a singular point of $f$ if and only if $|g(p)|=1$. 
\end{theorem}

We call $g$ the {\it Lorentzian Gauss map} of $f$. If $f$ has no singularities, then $g$ coincides with the composition of the 
Gauss map (i.e., (Lorentzian) unit normal vector) $n \colon \Sigma \to {\H}^{2}_{\pm}$ into the upper or lower connected component of 
the two-sheet hyperboloid ${\H}^{2}_{\pm}={\H}^{2}_{+}\cup {\H}^{2}_{-}$ in ${\R}^{3}_{1}$, where 
\begin{eqnarray}
{\H}^{2}_{+} &:=& \{n=(n^{1}, n^{2}, n^{3})\in {\R}^{3}_{1}\, ; \, \langle n, n \rangle = -1, n^{1}> 0\}, \nonumber \\ 
{\H}^{2}_{-} &:=& \{n=(n^{1}, n^{2}, n^{3})\in {\R}^{3}_{1}\, ; \, \langle n, n \rangle = -1, n^{1}< 0\},   \nonumber
\end{eqnarray}
and the stereographic projection from the north pole $(1, 0, 0)$ of the hyperboloid onto the Riemann sphere $\C\cup \{\infty\}$ (see \cite[Section 1]{UY2006}). 
A maxface is said to be {\it weakly complete} if the metric $d{\sigma}^{2}$ as in (\ref{max-nullmet}) is complete. 
We note that $(1/2)d{\sigma}^{2}$ coincides with 
the pull-back of the standard metric on ${\C}^{3}$ by the null holomorphic immersion of $f$ (see \cite[Section 2]{UY2006}). 

Applying Theorem \ref{main-1} to the metric $d{\sigma}^{2}$, we can get the following theorem. 
\begin{theorem}\label{max-Fujimoto}
Let $f\colon \Sigma \to {\R}^{3}_{1}$ be a maxface whose Lorentzian Gauss map $g\colon \Sigma \to \C\cup \{\infty\}$ omits more than $4$ $(=2+2)$ distinct 
values. Then there exists a positive constant $C$ depending on the set of exceptional values, but not $\Sigma$, such that for all $p\in \Sigma$ we have 
$$
|K_{d{\sigma}^{2}}(p)|^{1/2}\leq \dfrac{C}{d(p)}, 
$$
where $K_{d{\sigma}^{2}}(p)$ is the Gaussian curvature of the metric $d{\sigma}^{2}$ at $p$ and $d(p)$ is the geodesic distance from $p$ to  
the boundary of $\Sigma$. In particular, the Lorentzian Gauss map of a nonflat weakly complete maxface in ${\R}^{3}_{1}$ can omit at most $4$ $(=2+2)$ values.
\end{theorem}

As a corollary of this result, we give a simple new proof of the Calabi-Bernstein theorem (\cite{Ca1970}, \cite{CY1976}) for maximal space-like surfaces in 
${\R}^{3}_{1}$ from the viewpoint of value-distribution-theoretic properties of the Lorentzian Gauss map. We remark that Al\'ias and Palmer \cite{AP2001}, 
Estudillo and Romero \cite{ER1991, ER1992, ER1994}, Osamu Kobayashi \cite{Ko1983}, Romero \cite{Ro1996} and Umehara and Yamada \cite{UY2006} have 
approached this theorem from other viewpoints.

\begin{corollary}\label{Calabi-Bernstein}
Any complete maximal space-like surface in ${\R}^{3}_{1}$ must be a plane. 
\end{corollary}
\begin{proof}
Since a maximal space-like surface has no singularities, the complement of the image of $g$ contains at least the set $\{|g|=1\}\subset \C\cup \{\infty\}$. 
On the other hand, we obtain 
$$
ds^{2}=(1-|g|^{2})^{2}|\omega|^{2}\leq (1+|g|^{2})^{2}|\omega|^{2}=d{\sigma}^{2}. 
$$
Thus if $ds^{2}$ is complete, then $d{\sigma}^{2}$ is also complete. 
By Theorem \ref{max-Fujimoto}, its Lorentzian Gauss map is constant, that is, it is a plane. 
\end{proof}

\subsection{Lagrangian Gauss map of improper affine fronts in the affine 3-space}
Improper affine spheres in the affine $3$-space ${\R}^{3}$ also have similar properties to minimal surfaces in the Euclidean $3$-space. 
Recently, Mart\'inez \cite{Ma2005} discovered the correspondence between improper affine spheres and smooth special Lagrangian immersions in 
the complex $2$-space ${\C}^{2}$ and introduced the notion of {\it improper affine fronts}, that is, a class of (locally strongly convex) improper 
affine spheres with some admissible singularities in ${\R}^{3}$. We note that this class is called 
``improper affine maps'' in \cite{Ma2005}, but we call this class ``improper affine fronts'' because Nakajo \cite{Na2009} 
and Umehara and Yamada \cite{UY2011, UY2012} showed that all improper affine maps are wave fronts in ${\R}^{3}$. 
The differential geometry of wave fronts is discussed in \cite{SUY2009}. 
Moreover, Mart\'inez gave the following holomorphic representation for this class. 

\begin{theorem}\cite[Theorem 3]{Ma2005}
Let $\Sigma$ be a Riemann surface and $(F, G)$ a pair of holomorphic functions on $\Sigma$ such that $\text{Re}(FdG)$ is exact and 
$|dF|^{2}+|dG|^{2}$ is positive definite. Then the induced map $\psi\colon \Sigma \to {\R}^{3}=\C\times {\R}$ given by 
$$
\psi :=\biggl{(}G+\overline{F}, \dfrac{|G|^{2}-|F|^{2}}{2}+\text{Re}\biggl{(} GF- 2\int FdG \biggr{)} \biggr{)}
$$
is an improper affine front. Conversely, any improper affine front is given in this way. 
Moreover we set $x:= G+\overline{F}$ and $n:= \overline{F}-G$. Then $L_{\psi}:=x+\sqrt{-1}n\colon \Sigma \to {\C}^{2}$ is a special Lagrangian 
immersion whose induced metric $d{\tau}^{2}$ from ${\C}^{2}$ is given by 
$$
d{\tau}^{2}=2(|dF|^{2}+|dG|^{2}). 
$$
In addition, the affine metric $h$ of $\psi$ is expressed as $h:=|dG|^{2}-|dF|^{2}$ and the singular points of $\psi$ correspond to the 
points where $|dF|=|dG|$. 
\end{theorem}

The nontrivial part of the Gauss map of $L_{\psi}\colon \Sigma \to {\C}^{2}\simeq {\R}^{4}$ (see \cite{CM1987}) is the meromorphic function 
$\nu\colon \Sigma \to \C\cup\{\infty \}$ given by 
$$
\nu := \dfrac{dF}{dG}, 
$$
which is called the {\it Lagrangian Gauss map} of $\psi$. An improper affine front is said to be {\it weakly complete} if the induced metric 
$d{\tau}^{2}$ is complete. We note that 
$$
d{\tau}^{2}=2(|dF|^{2}+|dG|^{2})=2(1+|\nu|^{2})|dG|^{2}. 
$$

Applying Theorem \ref{main-1} to the metric $d{\tau}^{2}$, we can get the following theorem. This is a generalization of \cite[Theorem 3.2]{KN2012}. 

\begin{theorem}\label{affine-Fujimoto}
Let $\psi\colon \Sigma \to {\R}^{3}$ be an improper affine front whose Lagrangian Gauss map $\nu\colon \Sigma \to \C\cup \{\infty\}$ 
omits more than $3$ $(=1+2)$ distinct values. 
Then there exists a positive constant $C$ depending on the set of exceptional values, but not $\Sigma$, such that for all $p\in \Sigma$ we have 
$$
|K_{d{\tau}^{2}}(p)|^{1/2}\leq \dfrac{C}{d(p)}, 
$$
where $K_{d{\tau}^{2}}(p)$ is the Gaussian curvature of the metric $d{\tau}^{2}$ at $p$ and $d(p)$ is the geodesic distance from $p$ to  
the boundary of $\Sigma$. In particular, if the Lagrangian Gauss map of a weakly complete improper affine front in ${\R}^{3}$ is nonconstant, 
then it can omit at most $3$ $(=1+2)$ values. 
\end{theorem}

Since the singular points of $\psi$ correspond to the points where $|\nu|=1$, we can obtain a simple proof of the parametric affine Bernstein 
theorem (\cite{Ca1958}, \cite{Jo1954}) for improper affine spheres from the viewpoint of value-distribution-theoretic properties of the Lagrangian Gauss map. 
For the proof, see \cite[Corollary 3.6]{KN2012}. 

\begin{corollary}\label{affine-Bernstein}
Any affine complete improper affine sphere in ${\R}^{3}$ must be an elliptic paraboloid. 
\end{corollary}

\subsection{Ratio of canonical forms of flat fronts in the hyperbolic 3-space} 
We denote by ${\H}^{3}$ the hyperbolic $3$-space, that is, the simply connected Riemannian $3$-manifold with constant sectional curvature $-1$, 
which is represented as 
$$
{\H}^{3}= SL(2, \C) / SU(2) =\{aa^{\ast} \,;\, a\in SL(2, \C) \} \quad (a^{\ast}:=\tr{\bar{a}}). 
$$ 
For a holomorphic Legendrian immersion $\Lc\colon \Sigma \to SL(2, \C)$ on a simply connected Riemann surface $\Sigma$, the projection 
$$
f:= \Lc{\Lc}^{\ast}\colon \Sigma \to {\H}^{3}
$$
gives a {\it flat front} in ${\H}^{3}$. Here, flat fronts in ${\H}^{3}$ are flat surfaces in ${\H}^{3}$ with some admissible singularities  
(see \cite{KRUY2007}, \cite{KUY2004} for the definition of flat fronts in ${\H}^{3}$). We call $\Lc$ the {\it holomorphic lift} of $f$. 
Since $\Lc$ is a holomorphic Legendrian map, ${\Lc}^{-1}{d\Lc}$ is off-diagonal (see \cite{GMM2000}, \cite{KUY2003}, \cite{KUY2004}). 
If we set 
$$
{\Lc}^{-1}{d\Lc} = \left(
\begin{array}{cc}
0      & \theta  \\
\omega & 0
\end{array}
\right), 
$$
then the pull-back of the canonical Hermitian metric of $SL(2, \C)$ by $\Lc$ is represented as 
$$
ds^{2}_{\Lc}:=|\omega|^{2}+|\theta|^{2}
$$ 
for holomorphic $1$-forms $\omega$ and $\theta$ on $\Sigma$. A flat front $f$ is said to be {\it weakly complete} 
if the metric $ds^{2}_{\Lc}$ is complete \cite{KRUY2009, UY2011}. We define a meromorphic function on $\Sigma$ by the ratio of canonical forms 
$$
\rho := \dfrac{\theta}{\omega}. 
$$
Then a point $p\in \Sigma$ is a singular point of $f$ if and only if $|\rho (p)|=1$ \cite{KRSUY2005}. We note that 
$$
ds^{2}_{\Lc}=|\omega|^{2}+|\theta|^{2}=(1+|\rho|^{2})|\omega|^{2}. 
$$

Applying Theorem \ref{main-1} to the metric $ds^{2}_{\Lc}$, we can get the following theorem. This is a generalization of \cite[Theorem 4.5]{KN2012}. 
\begin{theorem}\label{flat-Fujimoto}
Let $f\colon \Sigma \to {\H}^{3}$ be a flat front on a simply connected  Riemann surface $\Sigma$. 
Suppose that the ratio of canonical forms $\rho\colon \Sigma \to \C\cup \{\infty\}$ omits more than $3$ $(=1+2)$ distinct values. 
Then there exists a positive constant $C$ depending on the set of exceptional values, but not $\Sigma$, such that for all $p\in \Sigma$ we have 
$$
|K_{ds_{\Lc}^{2}}(p)|^{1/2}\leq \dfrac{C}{d(p)}, 
$$
where $K_{ds^{2}_{\Lc}}(p)$ is the Gaussian curvature of the metric $ds^{2}_{\Lc}$ at $p$ and $d(p)$ is the geodesic distance from $p$ to  
the boundary of $\Sigma$. In particular, if the ratio of canonical forms of a weakly complete flat front in ${\H}^{3}$ is nonconstant, 
then it can omit at most $3$ $(=1+2)$ values. 
\end{theorem}

If $\Sigma$ is not simply connected, then we consider that $\rho$ is a meromorphic function on its universal covering surface $\widetilde{\Sigma}$. 
As a corollary of Theorem \ref{flat-Fujimoto}, we give a simple proof of the classification (\cite{Sa1973}, \cite{VV1971}) of complete nonsingular flat surfaces in $\H^{3}$. 
For the proof, see \cite[Corollary 3.5]{Ka2012}. 

\begin{corollary}\label{flat-Bernstein}
Any complete nonsingular flat surface in ${\H}^{3}$ must be a horosphere or a hyperbolic cylinder. 
\end{corollary}



\begin{thebibliography}{99}

\bibitem{Ah1935}
\textsc{L.~Ahlfors}, Zur Theorie der \"Uberlagerungsflachen, Acta Math. {\bf 65} (1935), 157--194.
\bibitem{AP2001}
\textsc{L.~J.~Al\'ias and B.~Palmer}, 
On the Gaussian curvature of maximal surfaces and the Calabi-Bernstein theorem, Bull. London Math. Soc. {\bf 33} (2001), 454--458. 
\bibitem{Ca1958}
\textsc{E.~Calabi}, 
Improper affine hypersurfaces of convex type and a generalization of a theorem by K.~J\"orgens, Mich. Math. J. {\bf 5} (1958), 108--126. 
\bibitem{Ca1970}
\textsc{E.~Calabi},  
Examples of Bernstein problems for some nonlinear equations, Proc. Sympos. Pure Math. {\bf 15} (1970), 223--230. 
\bibitem{Ch1960}
\textsc{S.~S.~Chern}, 
Complex analytic mappings of Riemann surfaces I., Amer. J. Math. {\bf 82} (1960), 323--337. 
\bibitem{CM1987}
\textsc{B.~Y.~Chen and J.~M.~Morvan}, G\'eom\'etrie des surfaces lagrangiennes de ${\C}^{2}$, J. Math. Pures Appl. {\bf 66} (1987), 
321--335. 
\bibitem{CY1976}
\textsc{S.~Y.~Cheng and S.~T.~Yau}, 
Maximal spacelike hypersurfaces in the Lorentz-Minkowski spaces, Ann. of Math. {\bf 104} (1976), 407--419. 
\bibitem{ER1991}
\textsc{F.~J.~M.~Estudillo and A.~Romero}, 
On maximal surfaces in the $n$-dimensional Lorentz-Minkowski space, Geom. Dedicate {\bf 38} (1991), 167--174. 
\bibitem{ER1992}
\textsc{F.~J.~M.~Estudillo and A.~Romero}, 
Generalized maximal surfaces in Lorentz-Minkowski space $L^{3}$, Math. Proc. Cambridge Philos. Soc. {\bf 111} (1992), 515--524. 
\bibitem{ER1994}
\textsc{F.~J.~M.~Estudillo and A.~Romero}, 
On the Gauss curvature of maximal surfaces in the $3$-dimensional Lorentz-Minkowski space, Comment. Math. Helv. {\bf 69} (1994), 1--4. 
\bibitem{Fu1988}
\textsc{H.~Fujimoto}, 
On the number of exceptional values of the Gauss map of minimal surfaces, J. Math. Soc. Japan {\bf 40} (1988), 235--247. 
\bibitem{Fu1993}
\textsc{H.~Fujimoto}, 
Value Distribution Theory of the Gauss Map of Minimal Surfaces in ${\R}^{m}$, Aspects of Mathematics, E21. Friedr. Vieweg \& Sohn, Braunschweig, 1993. 
\bibitem{Fu1997}
\textsc{H.~Fujimoto},
Nevanlinna theory and minimal surfaces, Geometry V, 95--151, 267--272, Encyclopaedia Math. Sci., {\bf 90}, Springer, Berlin, 1997. 
\bibitem{GMM2000}
\textsc{J.~A.~G\'alvez, A.~Mart\'inez and F.~Mil\'an}, Flat surfaces in hyperbolic $3$-space, Math. Ann. {\bf 316} (2000), 419--435. 
\bibitem{IK2008}
\textsc{T.~Imaizumi and S.~Kato}, 
Flux of simple ends of maximal surfaces in ${\R}^{2, 1}$, Hokkaido Math. J. {\bf 37} (2008), 561--610. 
\bibitem{Jo1954}
\textsc{K.~J\"orgens}, \"Uber die L\"osungen der differentialgleichung $rt-s^{2}=1$, Math. Ann. {\bf 127} (1954), 130--134. 
\bibitem{Ka2012}
\textsc{Y.~Kawakami}, A ramification theorem for the ratio of canonical forms of flat surfaces in hyperbolic three-space, 
preprint, arXiv:1110.3110. 
\bibitem{KKM2008}
\textsc{Y.~Kawakami, R.~Kobayashi and R.~Miyaoka}, 
The Gauss map of pseudo-algebraic minimal surfaces, Forum Math. {\bf 20} (2008), 1055--1069. 
\bibitem{KN2012}
\textsc{Y.~Kawakami and D.~Nakajo}, 
Value distribution of the Gauss map of improper affine spheres, 
J. Math. Soc. Japan {\bf 64} (2012), 799--821. 
\bibitem{Ko1983}
\textsc{O.~Kobayashi}, 
Maximal surfaces in the $3$-dimensional Minkowski space ${\mathbf{L}}^{3}$, Tokyo J. Math. {\bf 6} (1983), 297--309. 
\bibitem{Ko2003}
\textsc{R.~Kobayashi}, 
Toward Nevanlinna theory as a geometric model for Diophantine approximation, Sugaku Expositions {\bf 16} (2003), 39--79. 
\bibitem{KRSUY2005}
\textsc{M.~Kokubu, W.~Rossman, K.~Saji, M.~Umehara and K.~Yamada}, 
Singularities of flat fronts in hyperbolic space, Pacific J. Math. {\bf 221} (2005), 303--351. 
\bibitem{KRUY2007}
\textsc{M.~Kokubu, W.~Rossman, M.~Umehara and K.~Yamada}, Flat fronts in hyperbolic $3$-space and their caustics, 
J. Math. Soc. Japan {\bf 59} (2007), 265--299. 
\bibitem{KRUY2009}
\textsc{M.~Kokubu, W.~Rossman, M.~Umehara and K.~Yamada}, Asymptotic behavior of flat surfaces in hyperbolic $3$-space, 
J. Math. Soc. Japan {\bf 61} (2009), 799--852. 
\bibitem{KUY2003}
\textsc{M.~Kokubu, M.~Umehara and K.~Yamada}, An elementary proof of Small's formula for null curves in PSL(2, $\C$) and an 
analogue for Legendrian curves in PSL(2, $\C$), Osaka J. Math. {\bf 40} (2003), 697--715. 
\bibitem{KUY2004}
\textsc{M.~Kokubu, M.~Umehara and K.~Yamada}, Flat fronts in hyperbolic $3$-space, Pacific J. Math. {\bf 216} (2004), 149--175. 
\bibitem{Ma2005}
\textsc{A.~Mart\'inez}, 
Improper affine maps, Math.~Z. {\bf 249} (2005), 755--766. 
\bibitem{Na2009}
\textsc{D.~Nakajo}, A representation formula for indefinite improper affine spheres, Results Math. {\bf 55} (2009), 139--159. 
\bibitem{no1990} 
\textsc{J.~Noguchi and T.~Ochiai}, 
Geometric Function Theory in Several Complex Variables, Transl. Math. Monogr. {\bf 80}, Amer. Math. Soc., Providence, RI, 1990. 
\bibitem{Os1964}
\textsc{R.~Osserman}, 
Global properties of minimal surfaces in $E^{3}$ and $E^{n}$, Ann. of Math. {\bf 80} (1964), 340--364. 
\bibitem{Os1986}
\textsc{R.~Osserman}, 
A survey of minimal surfaces, second edition, Dover Publication Inc., New York, 1986. 
\bibitem{Ro1996}
\textsc{A.~Romero}, Simple proof of Calabi-Bernstein's theorem on maximal surfaces, Proc. Amer. Math. Soc. {\bf 124} (1996), 1315--1317. 
\bibitem{Ru2001}
\textsc{M.~Ru}, 
Nevanlinna theory and its Relation to Diophantine Approximation, World Sci., River Edge, NJ, 2001. 
\bibitem{SUY2009}
\textsc{K.~Saji, M.~Umehara and K.~Yamada}, 
The geometry of fronts, Ann. of Math. {\bf 169} (2009), 491--529. 
\bibitem{Sa1973}
\textsc{S.~Sasaki}, On complete flat surfaces in hyperbolic $3$-space, K\={o}dai Math Sem. Rep. {\bf 25} (1973), 449--457. 
\bibitem{UY2006}
\textsc{M.~Umehara and K.~Yamada}, 
Maximal surfaces with singularities in Minkowski space, Hokkaido Math. J. {\bf 35} (2006), 13--40. 
\bibitem{UY2011}
\textsc{M.~Umehara and K.~Yamada}, 
Applications of a completeness lemma in minimal surface theory to various classes of surfaces, 
Bull. London Math. Soc. {\bf 43} (2011), 191--199. 
\bibitem{UY2012} 
\textsc{M.~Umehara and K.~Yamada}, 
CORRIGENDUM: Applications of a completeness lemma in minimal surface theory to various classes of surfaces, 
to appear in Bulletin of the London Mathematical Society. (doi:10.1112/blms/bds017) 
\bibitem{VV1971}
\textsc{Y.~A.~Volkov and S.~M.~Vladimirova}, 
Isometric immersions of the Euclidean plane in Loba\u{c}evski\u{i} space (Russian), 
Mat.~Zametki {\bf 10} (1971), 327--332. 


\end{thebibliography}
\end{document}